\documentclass[12pt,dvips]{amsart}
\usepackage{euler, amsfonts, amssymb, latexsym, epsfig,graphicx} 

\usepackage{cancel}
\usepackage{pgfplots}
\usepackage[backref]{hyperref}
   \renewcommand*{\backrefalt}[4]{%
     \ifcase #1 %
     \or
     \vspace{0pt} \footnotesize (In \S#2)%
     \else
     \vspace{0pt} \footnotesize (In \S#2)
     \fi}
\renewcommand*{\backref}[1]{}

\newcommand\junk[1]{}
\junk{
  \usepackage[OT2,OT1]{fontenc}
  \newcommand\cyr{%
    \renewcommand\rmdefault{wncyr}%
    \renewcommand\sfdefault{wncyss}%
    \renewcommand\encodingdefault{OT2}%
    \normalfont
    \selectfont}
  \DeclareTextFontCommand{\textcyr}{\cyr}
}

\newtheorem{Theorem}{Theorem}
\newtheorem{Proposition}{Proposition}

\newtheorem{Corollary}{Corollary}
\newtheorem*{Corollary*}{Corollary}
\newtheorem*{Lemma*}{Lemma}

\newtheorem*{Theorem*}{Theorem}
\theoremstyle{remark}
\newtheorem{Example}{Example}

\newcommand\complexes{{\mathbb C}}
\newcommand\integers{{\mathbb Z}}
\newcommand\rationals{{\mathbb Q}}

\newcommand\naturals{{\mathbb N}}

\theoremstyle{plain}

\theoremstyle{remark}

\renewenvironment{quotation}
{\list{}{
    \setlength\itemindent{0em}%
    \setlength\leftmargin{1.5em}
    \setlength\rightmargin{1.5em}
  }%
\item[]}
{\endlist}

\newcommand\defn[1]{{\bf #1}} 

\newcommand\FF{{\mathbb F}}

\newcommand\wt{\widetilde}

\font\co=lcircle10

\def\jr{\smash{\raise2pt\hbox{\co \rlap{\rlap{\char'005} \char'007}}
               \raise6pt\hbox{\rlap{\vrule height6.5pt}}
               \raise2pt\hbox{\rlap{\hskip4pt \vrule height0.4pt depth0pt
                width7.7pt}}}}
\def\je{\smash{\raise2pt\hbox{\co \rlap{\rlap{\char'005}
                \phantom{\char'007}}}\raise6pt\hbox{\rlap{\vrule height6pt}}}}
\def\+{\smash{\lower2pt\hbox{\rlap{\vrule height14pt}}
                \raise2pt\hbox{\rlap{\hskip-3pt \vrule height.4pt depth0pt
                width14.7pt}}}}

\def\textcross{\ \smash{\lower4pt\hbox{\rlap{\hskip4.15pt\vrule height14pt}}
                \raise2.8pt\hbox{\rlap{\hskip-3pt \vrule height.4pt depth0pt
                width14.7pt}}}\hskip12.7pt}
\def\textelbow{\ \hskip.1pt\smash{\raise2.8pt%
                \hbox{\co \hskip 4.15pt\rlap{\rlap{\char'005} \char'007}
                \lower6.8pt\rlap{\vrule height3.5pt}
                \raise3.6pt\rlap{\vrule height3.5pt}}
                \raise2.8pt\hbox{%
                  \rlap{\hskip-7.15pt \vrule height.4pt depth0pt width3.5pt}%
                  \rlap{\hskip4.05pt \vrule height.4pt depth0pt width3.5pt}}}
                \hskip8.7pt}

\begin{document}
\pagestyle{plain}

\title{Randomly juggling backwards}
\dedicatory{to Ron Graham, former president of the International
  Juggler's Association, on his 80th birthday}

\author{Allen Knutson}
\thanks{AK was supported by NSF grant 0303523.}
\email{allenk@math.cornell.edu}

\maketitle

\renewcommand\AA{{\mathbb A}}
\newcommand\QQp{{\rationals_p}}
\newcommand\ZZp{{\integers_p}}

\begin{abstract}
  We recall the directed graph of \emph{juggling states}, closed walks within
  which give juggling patterns, as studied by Ron in \cite{Chung,Butler}.
  Various random walks in this graph have been studied before by
  several authors, and their equilibrium distributions computed. We
  motivate a random walk on the reverse graph (and an enrichment
  thereof) from a very classical linear algebra problem, leading to a
  particularly simple equilibrium: a Boltzmann distribution closely
  related to the Poincar\'e series of the $b$-Grassmannian in $\infty$-space.

  We determine the most likely asymptotic state in the limit of many balls, 
  where in the limit the probability of a $0$-throw is kept fixed.
\end{abstract}

{\Small
  \setcounter{tocdepth}{3}
  \tableofcontents
}

\section{Walks on the juggling digraph}\label{sec:intro}

The ``siteswap'' theory of juggling patterns was invented in
the early-mid '80s by Paul Klimek of Santa Cruz, Bruce Tiemann and
Bengt Magnusson at Caltech, 
and Mike Day and Colin Wright in Oxford. In 1988, having had some
time to digest this theory, Jack Boyce and I (each also at Caltech) 
independently invented a directed graph of ``juggling states'' 
to study siteswaps. We recall this definition now.

\subsection{The digraph}

Fix $b\in \naturals$ for the rest of the paper; it is called the
\defn{number of balls}. (Of course one might want a theory in which $b$ varies,
but we won't vary it in this paper.) A \defn{juggling state} $\sigma$
is just a $b$-element subset of $\naturals$, but we will draw it
as a semi-infinite word in $\times$ and $-$ using only $b$ many $\times$s,
e.g. $- \times \times - -\times - - - \ldots$. We won't generally
write any of the infinitely many $-$s after the last $\times$.
The physical interpretation of $\sigma$ is as follows: if a juggler
is making one throw each second, and we stop them mid-juggle and let the
$b$ balls fall to the ground, $\sigma$ records the sound ``wait,
thump, thump, wait, wait, thump'' (and thereafter, silence) that the balls make.
In the standard ``cascade'' pattern ($b$ odd) and asynchronous ``fountain'' 
pattern ($b$ even), this state is always the \defn{ground state} 
$\times \times \cdots \times$,
but other patterns go through more interesting states.\footnote{%
  E.g. the $3$-ball ``shower'' (juggling in a circle) alternates between
  the states $\times \times - \times$ and $\times - \times - \times$.}

\newcommand\sh{sh}

Put a directed edge $\sigma \to \tau$ if $\times \tau \supseteq \sigma$ 
(meaning, containment of the $\times$-locations). 
If the first letter of $\sigma$ is $\times$, then there is one
extra $\times$ in $\times \tau$ not in $\sigma$, in some position $t+1$;
call $t$ the \defn{throw} and label the edge with it. 
If the first letter of $\sigma$ is $-$, the throw is conventionally
taken to be $0$ (even though it's not much of a throw; the juggler
just waits for one beat).  \break 
If $\sigma \to \tau$, then any two of $(\sigma,\tau,t)$ determine the third.

\newcommand\To\longrightarrow
A closed walk in this digraph is called a \defn{juggling pattern} and
is determined by its sequence of throws, the \defn{siteswap}. 
Here is the (excellent) siteswap $501$ as a closed walk:
$$ \times - \times \qquad\stackrel{5}\To\qquad - \times - - \times
\qquad\stackrel{0}\To\qquad \times - - \times
\qquad\stackrel{1}\To\qquad \times - \times $$
Perhaps the earliest nontrivial theorem about siteswaps is
Ron et al.'s calculation that the number of patterns of length $n$
with at most $b$ balls is $(b+1)^n$ \cite{AMM}.
Ron and his coauthors have also counted
cycles that start from a given state \cite{Chung, Butler}. 

The first book on the general subject is \cite{Polster}.

In \cite{Warrington} was studied the finite subgraph 
in which only throws of height $\leq n$ are allowed,
giving ${n\choose b} = {n\choose n-b}$ states.
This digraph for $(b,n)$, with arrows reversed, 
is isomorphic to the one for $(n-b,n)$:
we reverse the length-$n$ states, and switch $\times$s and $-$s.
This remark serves as foreshadowing for the second paragraph below,
and the paper.

\subsection{A Markov chain}\label{ssec:basechain}

In \cite{Warrington, Varpanen1, Varpanen2, Varpanen3, Corteel1, Corteel2}
are studied Markov chains of juggling states, in which the possible throws
from a state are given probabilities. (Sometimes probability zero, making
them impossible; e.g. \cite{Warrington} puts a bound on the highest throw.)

We now define a Markov chain that follows the edges {\em backwards},
using a coin with $p($heads$) = 1/q$. 
{\bf However, we will never write the arrows as reversed:
  $\tau \to \tau'$ will have a consistent meaning throughout the paper.}
Let $\sigma$ be a juggling state.

\begin{enumerate}
\item Flip the coin at most $b$ times, or until it comes up tails.
\item If the coin never comes up tails, attach $-$ to the front of $\sigma$.
\item If the coin comes up tails on the $i$th flip, move the $i$th {\em last}
  $\times$ in $\sigma$ to the front, \\ leaving a $-$ in its place.
\end{enumerate}
Example: $\sigma = - -  \times \times -\, \times$, so $b=3$ and
we flip at most three times. If the flips are
\begin{itemize}
\item Tails: we get $\times - - \times \times - \cancel \times$, 
  i.e. $\times - - \times \times$.
\item Heads, then tails: we get $\times - - \times \cancel \times -  \times$, 
  i.e. $\times - - \times - - \times$.
\item Heads, heads, tails: we get $\times - - \cancel \times \times -  \times$, 
  i.e. $\times - - - \times - \times$.
\item Heads, heads, heads: we get $- - -  \times \times -  \times$.
\end{itemize}
Note that the resulting juggling states are exactly those that 
point to $\sigma$ in the digraph; \break the $\times$ that moves is
the ball thrown (if any; in the all-heads case the ``throw'' is a $0$).

Our main results in this paper are
\begin{itemize}
\item a calculation of the (quite simple) stationary distribution of
  this chain, 
\item a motivation and solution of the chain from linear algebra considerations,
  and
\item a study of the typical states in the $b\to\infty$ limit.
\end{itemize}

The limit $q\to\infty$ (always tails) is boring; after $b$ throws we
get to the ground state and stay there. The limit $q\to 1$ ($b$ fixed) has no
stationary distribution. In \S \ref{sec:stats} we show
the limit $b\to\infty, q\to 1$ is well-behaved if we keep fixed the
all-heads probability $E = q^{-b}$, which acts as a sort of temperature. 
Specifically, we compute 
the typical ball density around position $h$ to be $(1-E)/(1+E^{1-h/b}-E)$.

The linear algebra itself suggests in \S \ref{sec:richer} a Markov
chain on (the reverse of) a richer digraph with distinguishable balls
that can bump one another out of position; we solve this one as well.
(This digraph appeared first in \cite{Corteel2}, though we had been
considering it already for a few years; 
as far as we can tell our motivations for studying it 
are different than theirs.)

\subsubsection*{Acknowledgments} The author is very grateful to
Ron Graham for many conversations about mathematics and juggling,
and most especially for sending the author to speak about these subjects
in his stead at the Secondo Festival Della Matematica in Rome.\footnote{%
  This lecture is available at
  \url{https://www.youtube.com/playlist?list=PL3C8EC6BA111662D4} 
  in both English and Italian.}
Many thanks also to Jack Boyce, Svante Linusson, Harri Varpanen, 
and Greg Warrington.
Some related linear algebra was developed with David Speyer and Thomas Lam in
\cite{KLS}.

\section{The linear algebra motivation}\label{sec:motivation}

Let $U_b$ be the space of $b\times \naturals$ matrices of full rank, 
over the field $\FF$ with $q$ elements. Define a map
$$ \sigma:\ U_b \to \{\text{juggling states} \} $$
where there is a $\times$ in position $i$ of $\sigma\left(
  [\vec c_0 \vec c_1 \vec c_2 \cdots] \right)$
if $\vec c_i$ is not in the span of $\vec c_0, \ldots, \vec c_{i-1}$.
Equivalently, $\sigma(M)$ records the pivot columns in $M$'s
reduced row-echelon form.

\newcommand\dom\backslash

This $\sigma$ is preserved by and is the complete
invariant for $\{$row operations$\} \times \{$rightward column operations$\}$.
On the Grassmannian $GL(b) \dom U_b$ of $b$-planes in $\FF^\naturals$,
to which $\sigma$ descends, $\sigma$ records the 
(finite-codimensional) Bruhat cell of $rowspan(M)$.

We define a Markov chain on $U_b$, called ``add a random column $\vec c$ 
on the left'', meaning uniformly w.r.t. counting measure on $\FF^b$. 
Although this chain does not have an invariant probability distribution, 
it obviously preserves counting measure on $U_b$.

\begin{Proposition}\label{prop:Gtrans}
  Let $M \in U_b$, so $L = [\vec c\ M]$ is also in $U_b$.
  Then $\sigma(L) \to \sigma(M)$ in the juggling digraph.
  If we let $\tau$ range over the finite set of possible values of $\sigma(L)$,
  the probability of obtaining $\tau$ is given by the process
  described in \S \ref{ssec:basechain}.
\end{Proposition}

\begin{proof}
  For the first statement, we need only observe that if a column is
  pivotal in $[\vec c\ M]$, it is certainly pivotal in $M$.

  For the second, choose $j$ minimal such that $\vec c$ is 
  in the span of $M$'s left $j$ pivot columns.
  \begin{itemize}
  \item $j=0 \iff \vec c=0 \iff \sigma(L)$ is $\sigma(M)$ with a $-$ in front.
    Otherwise,
  \item $\sigma(L)$ is $\sigma(M)$ with its $j$th $\times$ moved to the front.
    There are $q^j$ $\vec c$s in the span of those $j$ columns,
    of which $q^{j-1}$ are in the span of the first $j-1$,
    for a probability of 
    $\left(q^j - q^{j-1}\right)/q^b = \left( 1-\frac{1}{q}\right)/q^{b-j}$,
    also the probability of tails after $b-j$ heads.
  \end{itemize}  
  Very similar results to the first statement appeared in 
  \cite{KLS} and \cite{Postnikov}, but about rotating the columns
  of a finite matrix.
\end{proof}

We now want to push this ``measure'' down to the set of juggling states,
i.e. for each juggling state $\tau$, we want to define
the probability that $M \in U_b$ has $\sigma(M) = \tau$.

\begin{Proposition}\label{prop:Gcount}
  Let $\tau$ be a juggling state, 
  and pick $N > $ the last $\times$-position of $\tau$.
  Then the fraction of $b\times N$ matrices with pivots in position $\tau$ is
  $$ 
  \frac{|GL_b(q)|}{q^{b^2}} 
  \ \bigg/\ q^{\ell(\tau)}
  \qquad \text{where }
  \ell(\tau) := { \#\{\text{``inversion" pairs }\ldots-\ldots\times\ldots\text{ in }\tau\}}
  $$
  independent of $N$.
\end{Proposition}

\begin{proof}
  Each such $M$ is row-equivalent, by a unique element of $GL_b(q)$,
  to a unique one in reduced row-echelon form. The pivotal columns in
  that are fixed (an identity matrix), accounting for the $q^{b^2}$ factor.
  With those columns erased, the remaining $b\times (N-b)$ matrix has a 
  partition's worth of $0$s in the lower left, and the complementary
  partition of free variables in the upper right. Each $0$ entry 
  corresponds to a pair $\ldots - \ldots \times \ldots$ in $\tau$.
\end{proof}

Rewriting the prefactor $  |GL_b(q)| \big/ {q^{b^2}} $, we get

\begin{Corollary}\label{cor:basedist}
  The mapping $\tau \mapsto \prod_{i=1}^b (1 - q^{-i}) \big/\ q^{\ell(\tau)}$
  is a probability measure on the space of juggling states 
  (meaning it sums to $1$, summing over all $\tau$).
\end{Corollary}

\begin{proof}
  Summing over only those $\tau$ with last $\times$ in position $\leq N$,
  we get the fraction of $b\times N$ matrices that are full rank.
  This goes quickly to $1$ as $N\to \infty$. 
\end{proof}

Put another way $ \prod_{i=1}^b (1-q^{-i})^{-1} = \sum_\tau q^{-\ell(\tau)} $,
which is easily justified for $q$ a formal variable: 
each side\footnote{%
  Matt Szczesny points out that this is a ``partition function''.
  If you don't get that joke be grateful.}
is a sum over Young diagrams with columns of height at most $b$,
of $q^{-\text{area}}$. The LHS computes this by counting how many columns 
of height $i$ there are, for each $i$.

The Weil conjectures relate counting points over $\FF$ to homology,
and each side of this equation is computing the Poincar\'e series of
the Grassmannian of $b$-planes in $\complexes^\naturals$.
This is closely related to Bott's formula for the Poincar\'e series of the 
{\em affine} Grassmannian, which also was related to juggling in \cite{ER}.

\begin{Theorem}\label{thm:basechain}
  This distribution in corollary \ref{cor:basedist} 
  is stationary for the Markov chain in \S \ref{ssec:basechain}.
\end{Theorem}

We could likely justify this for $q$ a prime power through some 
limiting procedure in $N$, but it's easy enough to check for 
a formal variable $q$ (i.e. $q^{-\infty} = 0$), so we do that now.

\begin{proof}
  Stationarity at $\tau$ says
  $$ \sum_{\tau \to \tau'} p(\tau') p(\tau',\tau) = p(\tau) $$
  (here $\to$ indicates the edge in the usual juggling digraph,
  whereas $p(\tau',\tau)$ is the transition probability calculated in
  proposition \ref{prop:Gtrans}). If $\tau$ begins with $-$, 
  then the only $\tau'$ is $\tau$ with that $-$ removed, 
  and stationarity says
  \begin{eqnarray*}
     p(\tau') q^{-b} &=& p(\tau) \\
     \prod_{i=1}^b (1-q^{-i})^{-1}\ q^{-\ell(\tau')} q^{-b} 
     &=& \prod_{i=1}^b (1-q^{-i})^{-1}\ q^{-\ell(\tau)}
  \end{eqnarray*}
  or $\ell(\tau')+b = \ell(\tau)$, which is obvious from the definition 
  of $\ell$.

  If $\tau$ begins with $\times$ (i.e. we have a ball to throw), then
  there are infinitely many $\tau'$ it could throw to. We group
  $\tau$'s $-$s into $b$ many groups $j \in [1,b]$, where $j$ is the
  number of $\times$ in $\tau$ that the throw skips past (counting
  itself, hence $j\geq 1$).  Let $\lambda_j$ be the position of the
  $j$th $\times$ in $\tau$, with $\lambda_1 = 0$, $\lambda_{b+1} := \infty$. Then
  \begin{eqnarray*}
    \prod_{i=1}^b (1-q^{-i})^{-1} \sum_{\tau \to \tau'} p(\tau') p(\tau',\tau) 
    &=& \sum_{\tau \to \tau'} q^{-\ell(\tau')} p(\tau',\tau) 
    \ =\ \sum_{j=1}^b \sum_{t \in (\lambda_{j},\lambda_{j+1}) \atop
      \tau' := \tau\text{ after $t$-throw}} q^{-\ell(\tau')} p(\tau',\tau) \\
    &=& \sum_{j=1}^{b} \sum_{t \in (\lambda_{j},\lambda_{j+1}) \atop
      \tau' := \tau\text{ after $t$-throw}}  q^{-\ell(\tau')} q^{j-b} (1 - q^{-1}) \\
    &=& q^{-b} \sum_{j=1}^{b} q^j (1 - q^{-1})
    \sum_{t \in (\lambda_{j},\lambda_{j+1}) \atop
      \tau' := \tau\text{ after $t$-throw}}     q^{-\ell(\tau')} \\
  \end{eqnarray*}
  To make $\tau'$, we remove the $\times$ from the front of $\tau$
  (preserving $\ell$), and put it in position $t-1 \in \naturals$, 
  which places it right of $t-1$ other letters of which $j-1$ are $\times$.
  That creates $t-j$ of the $-\times$ inversions to its left, 
  while destroying $b-j$ inversions from its right:
  \begin{eqnarray*}
    q^{-b} \sum_{j=1}^{b} q^j
    (1 - q^{-1}) \sum_{t \in (\lambda_{j},\lambda_{j+1}) \atop
      \tau' := \tau\text{ after $t$-throw}}     q^{-\ell(\tau')} 
    &=& q^{-b} \sum_{j=1}^{b} q^j (1 - q^{-1}) 
    \sum_{t \in (\lambda_{j},\lambda_{j+1}) \atop
      \tau' := \tau\text{ after $t$-throw}}     q^{-(\ell(\tau)+(t-j)-(b-j))} \\    
    &=&  q^{-\ell(\tau)} \sum_{j=1}^{b} q^j (1 - q^{-1}) 
    \sum_{-t \in (-\lambda_{j+1},-\lambda_{j}) }     q^{-t} \\    
\text{which telescopes as}
    &=&  q^{-\ell(\tau)} \sum_{j=1}^{b} q^j (q^{-\lambda_{j}-1} - q^{-\lambda_{j+1}}) \\
\text{which in turn telescopes as}
    &=&  q^{-\ell(\tau)} (q^1 q^{-\lambda_1-1} - q^b q^{-\lambda_{b+1}}) \\
    &=& q^{-\ell(\tau)} (q^1 q^{-0-1} - q^b q^{-\infty}) \qquad = q^{-\ell(\tau)}.
  \end{eqnarray*}
  \vskip -.3in
\end{proof}

\emph{Two comments.}
Another way we could have made rigorous the probability measure 
on matrices, and then pushed it down to the set of states, would
be to group matrices into equivalence classes where $M \sim M'$ if
they have the same pivot columns, and agree in the columns
up to and including their $b$th pivot column. 

Also, instead of working
with full-rank matrices we could have worked with all matrices,
allowing $b' < b$ pivot columns. On the level of juggling states,
this amounts to having the remaining $b-b'$ many $\times$s sitting in abeyance
at the (infinite) right end of the state; in short order those $\times$s
move to finite positions and never go back. This larger Markov chain
is not ergodic, and the $b'<b$ states have $0$ probability, so we just
left them out for ease of exposition.

We record for later use this function
$$ s_n = s_n(q) := \prod_{i=1}^n (1-q^{-i}) $$
whose reciprocal is $\sum_{\text{$n$-state }\tau} q^{-\ell(\tau)}$,
and connect it to other well-known Poincar\'e series:

\begin{Proposition}\label{prop:series}
  The Poincar\'e series of the flag manifold $Fl(n)$, 
  in the nontraditional variable $q^{-1}$, is 
  ${s_n} / {s_1^n} = \sum_{\pi \in S_n} q^{-\ell(\pi)}$. 
  The Poincar\'e series of the Grassmannian $Gr(j,h)$ 
  is ${s_h} / {s_j s_{h-j}} = \sum_{\sigma \in {h\choose j}} q^{-\ell(\sigma)}$.
  The second sum is over $j$-ball juggling states with no $\times$s in 
  position $h$ or later.
\end{Proposition}

\section{The $b\to\infty$ limit, with fixed 
  probability of placing an initial ``$-$''}\label{sec:stats}

How many $\times$s are we most likely to have in the first $h$ spots? 

The trick we use to measure such states is to look 
at concatenations $\sigma_L \sigma_R$ where $\sigma_L$ is a finite
string with multiplicities $\times^c -^{h-c}$, 
and $\sigma_R$ an infinite one with multiplicities $\times^{b-c} -^\infty$.
Then 
$$ \ell(\sigma_L \sigma_R) = \ell(\sigma_L)+\ell(\sigma_R)+(b-c)(h-c) $$
as the third term counts the inversions of the $(h-c)$ $-$s in
$\sigma_L$ with the $(b-c)$ $\times$s in $\sigma_R$.

Using this and proposition \ref{prop:series}, we calculate
the probability $P_c$ of having exactly \break 
$c \in [0,\min(h,b)]$ $\times$s in $[0,h-1]$ as
\begin{eqnarray*}
  \sum_{\sigma_L \in {[0,h-1]\choose c}} \sum_{\sigma_R}
  s_b q^{-\ell(\sigma_L)-\ell(\sigma_R)-(b-c)(h-c)}
  &=&
  s_b q^{-(b-c)(h-c)}
  \sum_{\sigma_L \in {[0,h-1]\choose c}} q^{-\ell(\sigma_L)} \sum_{\sigma_R} q^{-\ell(\sigma_R)} \\
  &=&
  s_b q^{-(b-c)(h-c)} \frac{s_h}{s_c s_{h-c}} \frac{1}{s_{b-c}}
\end{eqnarray*}
which is maximized at the $c$ where $P_c/P_{c-1}$ crosses from $>1$ to $<1$.
That ratio is
\begin{eqnarray*}
  P_c / P_{c-1} &=& \frac{q^{-(b-c)(h-c)}}{q^{-(b-c+1)(h-c+1)}} 
  \frac{s_{c-1}}{s_c} \frac{s_{h-c+1}}{s_{h-c}} \frac{s_{b-c+1}}{s_{b-c}} \\ \\
  &=& q^{-c+h+b-c+1} (1-q^{-c})^{-1} (1-q^{-(h-c+1)}) (1-q^{-(b-c+1)}) \\ 
  &=&  (q^{c}-1)^{-1} (q^h-q^{c-1}) (q^{b-c+1}-1) 
\end{eqnarray*}
and setting it to $1$ gives
\begin{eqnarray*}
  q^{c}-1 &=& (q^h-q^{c-1}) (q^{b-c+1}-1) \\
  \left( q^{c}-1 \right) \big/ \left( q^{b-c+1}-1 \right) &=& q^h-q^{c-1} \\
  q^{c-1} + \frac{q^{c}-1}{q^{b-c+1}-1} &=& q^h 
\end{eqnarray*}
Toward considering the $b\to \infty$ limit, 
let $\lambda := c/b \in [0,1]$, $\mu := h/b \in [\lambda,\infty)$, 
and $E = q^{-b}$:
$$ q^{-1} E^{-\lambda} + \frac{E^{-\lambda}-1}{q E^{\lambda-1} - 1} = E^{-\mu} $$

This $E$ (for ``empty hand'') is the probability of never flipping tails,
thereby putting a $-$ at the front of the state.
The limit $q\to 1, E\to 1$ of backwards juggling is thus not interesting.
Instead, we consider simultaneous limits $q\to 1,b\to \infty$
in such a way that $E$ has a limit in $(0,1)$, e.g. $q = 1 - 1/b$. 
Then
\begin{eqnarray*}
  E^{-\mu} &=& E^{-\lambda} + \frac{E^{-\lambda}-1}{E^{\lambda-1} - 1} 
  = E^{-\lambda} \left(1 + \frac{1-E^{\lambda}}{E^{\lambda-1} - 1}\right) 
  = E^{-\lambda}\   \frac{E^{\lambda-1}-E^{\lambda}}{E^{\lambda-1} - 1} 
  = E^{-\lambda}\   \frac{1-E}{1-E^{1-\lambda}} \\
  \mu &=& \lambda + \log_{E^{-1}} \frac{1-E}{1-E^{1-\lambda}} 
  \qquad\qquad \text{note }E^{-1} > 1 > E^{1-\lambda} > E > 0,
  \text{ so }\mu>\lambda   \\
  &\sim& \lambda + \frac{E^{1-\lambda}}{\log(E^{-1})} \quad \text{as }E\to 0,
  \lambda\text{ fixed, or} \qquad 
  \frac{\log(1-\lambda)^{-1}}{1-E}\quad\text{as }E\to 1,\lambda\text{ fixed}
\end{eqnarray*}
To recap: if we consider the limit $b\to\infty$ of many balls,
and don't control $E$, then in the $E\to 0$ limit we get $\mu = \lambda$,
the ground state. The limit $E\to 1$ doesn't exist. But if $E \in (0,1)$,
then the function $\mu(\lambda)$ above says how far out (as a multiple
of $b$) one should look to find the first $\lambda$ balls (as a
fraction of $b$). 

We can invert this relation to find $\lambda$ in terms of $\mu$:
\begin{eqnarray*}
  E^{-\mu} (1 - E^{1-\lambda}) &=& E^{-\lambda} (1-E) \\
  E^{-\mu} - E^{1-\mu} E^{-\lambda} &=& E^{-\lambda} (1-E) \\
  E^{-\mu}  &=& E^{-\lambda} (1-E + E^{1-\mu}) \\
  E^{-\mu} (1-E + E^{1-\mu})^{-1} &=& E^{-\lambda}  \\
  \mu - \log_{E^{-1}} (1+ E^{1-\mu}-E ) &=& \lambda
  \qquad\qquad \text{note }E^{1-\mu} > E, \text{ so } \lambda<\mu
\end{eqnarray*}
For example, as $E\to 0$ the fraction $\lambda$ of balls in 
the first $b$ slots is $1 - \log_{E^{-1}}(2-E) \sim 1 - \ln 2 / ln (E^{-1})$,
i.e. all but $\ln 2/\ln(E^{-1})$.
\junk{
Whereas for $E\to 1$, the fraction is
\begin{eqnarray*}
   1 - \frac{\log (1+ E^0-E )}{\log (E^{-1})} 
   = 1 +  \frac{\log (2-E )}{\log E} 
   \sim 1 +  \frac{-1/(2-E)}{1/E} = \frac{(2-E)+E}{2-E} 
   = \frac{1}{1-E/2} \sim 1-E/2
\end{eqnarray*}
\begin{eqnarray*}
  1 + \frac{\log(1+\varepsilon)}{\log (1-\varepsilon)} 
  \sim 1 + \frac{1/(1+\varepsilon)}{-1/(1-\varepsilon)} 
  = 1 - \frac{1-\varepsilon}{1+\varepsilon} 
\end{eqnarray*}

Whereas for $E = 1-\varepsilon$, the fraction is 
\begin{eqnarray*}
 1 - \frac{\log(2-(1-\varepsilon))}{\log (1-\varepsilon)^{-1}}
 &=& 1 + \frac{\log(1+\varepsilon)}{\log (1-\varepsilon)} 
 \sim 1 + \frac{\varepsilon - \varepsilon^2/2 + \ldots}
 {-\varepsilon - \varepsilon^2/2 + \ldots} 
 = 1 + \frac{1 - \varepsilon/2 + \ldots}  {-1 - \varepsilon/2 + \ldots} \\
 &=& 1 - \frac{1 - \varepsilon/2 + \ldots}  {1 + \varepsilon/2 + \ldots} 
 \sim 1 + (1-\varepsilon) = \varepsilon
\end{eqnarray*}
}
\junk{
If only $\varepsilon$ of the balls are in $[0,1]$ (in units of $b$),
it suggests we ask what fraction of the balls are 
in the range $[0,\mu = 1/\varepsilon]$.
\begin{eqnarray*}
  \frac{1}{\varepsilon} 
  - \frac{\log(1 + (1-\varepsilon)^{1-1/\varepsilon} - (1-\varepsilon))}
  {\log(1+\varepsilon)}
  &=&    \frac{1}{\varepsilon} 
  - \frac{\log((1-\varepsilon) (1-\varepsilon)^{-1/\varepsilon} + \varepsilon)}
  {\log(1+\varepsilon)} \\
  &\sim& \frac{1}{\varepsilon} 
  - \frac{\log(e (1-\varepsilon) + \varepsilon)}
  {\varepsilon} \\
\end{eqnarray*}
}

\newcommand\ddmu{\frac{d}{d\mu}}
The {\em ball density} at $\mu$ is the derivative of this $\lambda(\mu)$ 
w.r.t. $\mu$,
\begin{eqnarray*}
  1 - \frac{\ddmu \log(1+E^{1-\mu}-E)} {\log(E^{-1})}
  &=& 1 - \frac{\ddmu (1+E^{1-\mu}-E)} {\log(E^{-1}) (1+E^{1-\mu}-E)} 
  = 1 - \frac{\ddmu (E^{1-\mu})} {\log(E^{-1}) (1+E^{1-\mu}-E)} \\
  &=& 1 - \frac{\log(E^{-1}) (E^{1-\mu})} {\log(E^{-1}) (1+E^{1-\mu}-E)} 
  = 1 - \frac{E^{1-\mu}} { 1+E^{1-\mu}-E} \\
  &=& \frac{1-E} { 1+E^{1-\mu}-E} 
\end{eqnarray*}
which is $1-E$ at $\mu=0$ (as befits the definition of $E$)
and decreases thereafter. 
As another sanity check, consider the limit $E\to 0$ with $\mu$ fixed:
for $\mu<1$ we get $\frac{1-0}{1+0-0} = 1$, 
whereas for $\mu>1$ we get $\frac{1-0}{1+\infty-0} = 0$.
See figure \ref{fig:density} on p\pageref{fig:density}.

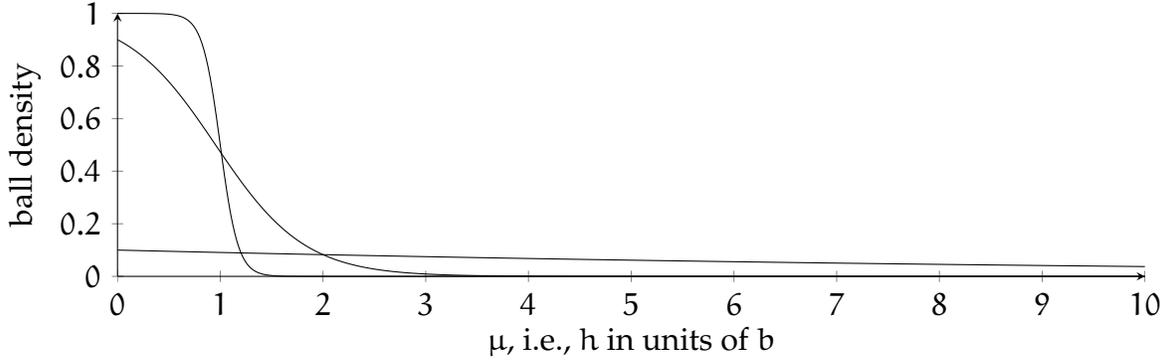
\begin{figure}[htbp]
  \centering
  \begin{tikzpicture}
    \begin{axis}[domain = 0:6, width = 6in, height = 2in, samples = 500,
      axis x line = bottom,
      axis y line = left,
      ylabel = ball density, xlabel = {$\mu$, i.e., $h$ in units of $b$}  ] 
      \addplot[domain=0:10,]  {(1-0.00001)/(1 + (.00001)^(1-x) - .00001)};
      \addplot[domain=0:10,]  {(1-0.1)/(1 + (.1)^(1-x) - .1)};
      \addplot[domain=0:10,]  {(1-0.9)/(1 + (.9)^(1-x) - .9)};
    \end{axis}
  \end{tikzpicture}
  \caption{The ball density functions for $E = 0.00001$ (the sigmoidal curve),
    $E = 0.9$ (the very flat one at the bottom), and $E=0.1$ in between;
    each $y$-intercept is $1-E$.
    Recall that the fraction in the tail $\mu\geq 1$ is about
    $\ln 2/\ln(E^{-1})$, rather a lot, which is why we need $E$ so very small 
    to get a sigmoidal-looking curve.}
  \label{fig:density}
\end{figure}

\section{Some richer linear algebra, and flag juggling}
\label{sec:richer}

As explained at the beginning of \S \ref{sec:motivation}, the function
$\sigma$ was the complete invariant for the group of row operations and 
rightward column operations. If we restrict to {\em downward} row operations, 
then we still get a discrete set of orbits (even for complex matrices); 
each orbit contains a unique partial permutation matrix of rank $b$.

Define a \defn{flag juggling state} as a juggling state where the
$\times$s have been replaced by the numbers $1,\ldots,b$,
each used exactly once. Then we have a unique map 
$\wt\sigma:\ U_b \to \{$flag juggling states$\}$
that takes a partial permutation matrix of rank $b$ with $m_{ij} = 1$
to a state with an $i$ in the $j$th position,
and such that $\sigma$ is invariant under 
downward row and rightward column operations.

To give the analogue of proposition \ref{prop:Gcount}
requires us to extend the definition of $\ell$ to flag juggling
patterns: it should also count any pair $\ldots i \ldots j \ldots$
with $i>j$ as an inversion, e.g. $\ell(-\ 3-1\ 2) = 7$.
It is then reasonable to consider ``$-$'' as $+\infty$ for this inversion count.

\begin{Proposition}\label{prop:Fcount}
  Let $\wt\tau$ be a flag juggling state, 
  and pick $W > $ the last $\times$-position of $\tau$. \\
  Then the fraction of $b\times W$ matrices with $\wt\sigma(M) = \wt\tau$ is
  $\left( 1 - q^{-1} \right)^b \ \big/\ q^{\ell(\wt\tau)}$, independent of $W$.
\end{Proposition}

\begin{proof}
%
  We're computing the size of the $B_- \times N_+$-orbit through the
  partial permutation matrix $M$ with $\wt\sigma(\pi) = \wt\tau$, 
  where $B_-$ is
  lower triangular $b\times b$ matrices and $N_+$ is upper triangular
  $W\times W$ matrices with $1$s on the diagonal. 

  The $N_+$-stabilizer of $M$ consists of matrices $R$ with $R_{ij} = 0$
  unless $\wt\tau$ has $-$ in its $i$th position. The $B_-$-stabilizer
  is trivial.
  However, the $(B_- \times N_+)$-stabilizer of $M$ is slightly larger
  than the product of the stabilizers; 
  some row operations can be canceled by some column operations, 
  one such pair for each inversion $\ldots i \ldots j \ldots$ with $i>j$.

  The order of $B_- \times N_+$ is $(q-1)^b q^{b\choose 2} q^{W\choose 2}$;
  dividing by the stabilizer order gives the size of the orbit,
  then by $q^{bW}$ gives the fraction claimed.
\end{proof}

\begin{Corollary}\label{cor:flagdist}
  The mapping 
  $\wt\tau \mapsto   \left( 1 - q^{-1} \right)^b \ \big/\ q^{\ell(\wt\tau)} $
  is a probability measure on the space of flag juggling states.  
\end{Corollary}

{\em Side note.}
The corresponding equation $(1-q^{-1})^{-b} = \sum_{\wt\tau} q^{-\ell(\wt\tau)}$
gives two formulae for the Poincar\'e series of the manifold
$B_- \dom U_b$ of partial flags $(V^1 < V^2 < \ldots< V^b < \complexes^\naturals)$.
Since $B_- \dom U_b$ is a Leray-Hirsch-satisfying bundle over
$GL(b) \dom U_b$ with fiber $B_- \dom GL(b)$, the Poincar\'e series
$b_{B_- \dom U_b}$ of this bundle factors as
$$ b_{B_- \dom U_b} = b_{GL(b) \dom U_b}\ b_{B_- \dom GL(b)} $$
where
\begin{eqnarray*}
  b_{B_- \dom U_b} &=& (1-q^{-1})^{-b} = \sum_{\wt\tau} q^{-\ell(\wt\tau)} \\
  b_{GL(b) \dom U_b} &=& \prod_{i=1}^b (1-q^{-i})^{-1} = \sum_\tau q^{-\ell(\tau)} \\
  b_{B_- \dom GL(b)} &=& \prod_{i=1}^b \frac{1-q^{-i}}{1-q^{-1}} = \sum_{\pi \in S_n} q^{-\ell(\pi)} 
\end{eqnarray*}
the three sums on the right
being derivable from the respective Bruhat decompositions.

We now define the edges out of a flag juggling state $\wt\tau$,
again making the set of states into the vertices of a digraph
which appeared already in \cite{Corteel2}.
(As far as we can tell our motivations for studying this digraph
are different than theirs.)
A small example is shown in figure \ref{fig:flagdigraph}.

If $\wt\tau$ begins with $-$, then there is a unique outgoing edge,
to $\wt\tau$ with the $-$ removed. 
Otherwise we pick up the number that $\wt\tau$ starts with and
begin walking East.
\begin{enumerate}
\item At any $-$, we can replace the $-$ with the carried number and be done.
\item At any strictly larger number, we can pick up that larger number,
  drop the number we were carrying in its place, and go back to (1).
\end{enumerate}
Define the \defn{throw set} for the transition $\wt\tau \to \wt\tau'$
as the places a number is dropped.
Neither drop is required; we can continue walking East instead
(though not forever). Note that if we used the label $1$ $b$ times
instead of $[1,b]$ each once, then each throw set would be singleton,
and this would be the same digraph as in \S \ref{sec:intro}.

\begin{figure}[hbtp]
  \centering
  \epsfig{file=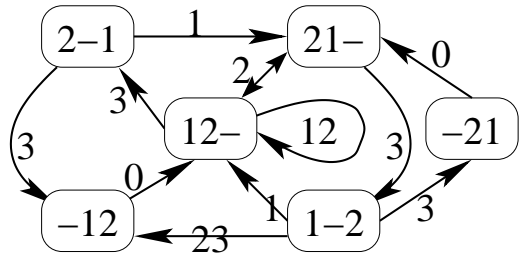, height=1.3in}
  \caption{The $b=2$ flag juggling digraph with 
    and only throws $\leq 3$ drawn.
  }
  \label{fig:flagdigraph}
\end{figure}

\subsection{Another Markov chain}

As in \S\ref{ssec:basechain}, we define a Markov chain following the
edges backwards in this digraph. Let $\wt\tau$ be a flag juggling state,
and again we use a coin with $p($heads$)=1/q$. 
\begin{enumerate}
\item Hold a $-$, and point at the rightmost number in $\wt\tau$.
\item Flip the coin. If tails, put down what we're holding and pick
  up the number we're pointing at. If heads, do nothing.
\item Move leftwards; stop when we meet
  a number smaller than what we're holding (interpreting $-$ as $+\infty$,
  jibing with our definition of $\ell(\wt\tau)$).
\item If we meet such a number, go back to (2). Otherwise we've fallen
  off the left end of (the now modified) $\wt\tau$; drop whatever 
  we're holding, there.        
\end{enumerate}

For example, start with $\wt\tau = - -\ 3\ 1 - 2$, holding a $-$,
pointing at the $2$.
\begin{itemize}
\item Tails: pick up the $2$, leaving the $-$ in its place. Point at the $1$.
  \begin{itemize}
  \item Tails: drop the $2$ for the $1$, then carried all the way left
    to give $1 - -\ 3\ 2 - -$.
  \item Heads: the $2$ gets carried all the way left
    to give $2 - -\ 3\ 1 - -$.
  \end{itemize}
\item Heads: skip the $2$ and point at the $1$.
  \begin{itemize}
  \item Tails: pick up the $1$, leaving the $-$, and carry the $1$ all
    the way left to give $1 - -\ 3- -\ 2$.
  \item Heads: leave the $1$ and proceed to the $3$.
    \begin{itemize}
    \item Tails: pick up the $3$ and carry it left to give $3 - - - 1 - 2$.
    \item Heads: leave the $3$ and drop the $-$ on the left, 
      giving $---\ 3\ 1-2$.
    \end{itemize}
  \end{itemize}
\end{itemize}
\begin{eqnarray*}
\text{In all, }
   \wt\tau &\mapsto& (1-q^{-1})^2[1--3\ 2] 
   \ +\  (1-q^{-1})q^{-1} \left( [2--3\ 1--] + [1--3--2] \right) \\
   && + q^{-2} (1-q^{-1}) [3---1-2] \ +\ q^{-3} [---3\ 1-2] 
\end{eqnarray*}

This is again motivated by the ``add a random column on the left''
Markov chain on $U_b$:

\begin{Proposition}
  Let $M$ be the partial permutation matrix in $U_b$ 
  with $\wt\sigma(M)=\wt\tau$, and $\vec c$ a random column vector. 
  Then the probability of $\wt\sigma([\vec c M])$ being a particular
  state $\tau'$ is the probability of reaching $\tau'$ in the
  process above.  

  In particular, the possible $\tau'$ are the ones such that 
  $\tau' \to \tau$ in the digraph defined in this section.
\end{Proposition}

\begin{proof}[Proof sketch]
  Rightward column-reduction of $[\vec c\ M]$ corresponds to doing 
  the coin flips, in reverse order.
  At each coin-flipping step, we determine that a certain entry of
  $\vec c$ is zero (heads) or nonzero (tails). 
  We leave the details to the reader.
\end{proof}

We have the corresponding theorem, but skip the 
corresponding formal derivation:

\begin{Theorem}
  The vector $\wt\tau \mapsto (1-q^{-1})^b q^{-\ell(\wt\tau)}$ 
  is the stationary distribution of the Markov chain (1)-(4) defined above.
\end{Theorem}

\subsection{The linear algebra of repeated labels}

Given a finite multiset $S$ of numbers, we can redefine \defn{flag juggling
  states} to bear those labels, and the Markov chain in this section
extends without changing a word. If the elements of $S$ are all equal, or all
different, we get the digraphs from \S \ref{sec:intro} and \S \ref{sec:richer}
respectively. So one can ask for a corresponding linear algebra problem
and, hopefully thereby, calculation of the stationary distribution.

To interpolate between {\em all} row operations and
{\em downward} row operations, we consider the rows as coming in
contiguous groups, and only allow row operations within a group or
downward. For example if $S = \{1,1,5,7,7,7\}$, then we have three
groups, which are two rows above one row above three rows.

In the analogue of reduced row-echelon form, the $b$ pivot columns are
still arbitrary, but the pivots within a given group run Northwest/Southeast.
In the analogue of corollaries \ref{cor:basedist} and \ref{cor:flagdist},
the prefactor is
$$ \prod^\text{groups} \prod_{i=1}^\text{group size} (1-q^{-i}) $$
which computes the probability that a {\em block} lower triangular matrix
is invertible. These prefactors, times $q^{-\ell(\wt\tau)}$, 
again give the stationary distribution.

\section{Drawing out the process}\label{sec:hats}

Since the transitions in the (flag) juggling Markov chain involve 
repeated flipping of coins, it seems natural to ask for an
alternative version with intermediate states, such that only one
coin is flipped at each transition.

\subsection{The digraph of ordinary and hatted states}

\newcommand\qtoq{\quad\to\quad}

Define a \defn{hatted flag juggling state} $\tau_+$
as a flag juggling state $\tau$ with one position hatted,
as in $\hat 5$ or $\hat -$. The hat is not allowed over one of the $-$s
occurring after $\tau$'s last number, e.g. $3 - 4 - -\ldots$ can
only be hatted as $\hat 3 - 4$, $3 \hat- 4$, or $3 - \hat 4$.

The vertices of the digraph will be the usual unaugmented states, 
plus these new, \break intermediate, states. Make directed edges as follows:
\begin{itemize}
\item If $\tau$ is unhatted, then it has only one arrow out, hatting 
  the $0$th label. \\ Example: $\qquad 3 -\ 2\ 1 \qtoq \hat 3 -\ 2\ 1$.
\item If $\tau_+$ has a hat, then there are one or two arrows out of $\tau_+$.
  \begin{itemize}
  \item We can move the hat and label one step rightward, switching 
    places with the unhatted label just beyond, unless that involves
    switching a $\hat -$ with the last number in $\tau$. \
    Example: \qquad $\hat 3 -\ 2\ 1 \qtoq  -\ \hat 3\ 2\ 1 $.
  \item If the hatted label is the last number in $\tau$, we can 
    remove the hat.
    If it isn't, and the next label after the hatted label is larger
    (counting $-$ as $\infty$), 
    then the hat can jump one step rightward to that larger label,
    with the labels not moving. \\
    Example: \qquad $\hat 3 -\ 2\ 1 \qtoq 3\   \hat-\ 2\ 1 
    \quad\cancel\to\quad 3 -\hat 2\ 1 $.
  \end{itemize}
\end{itemize}

It's easy to see that if $\tau, \tau'$ are unhatted states, 
then $\tau \to \tau'$ in the digraph from \S \ref{sec:richer}
iff there is a directed path in this digraph from $\tau$ to $\tau'$ 
through only hatted states.

\subsection{The last Markov chain}\label{ssec:fiddlychain}

As twice before, we put probabilities on the reversed edges.

If $\tau$ is an unhatted state, the only edge into it has a hat on
the last number in $\tau$.
Whereas if $\tau_+$ is $\tau$ with a hat at position $0$, 
the only edge into $\tau_+$ comes from $\tau$.
In either case the unique edge gets probability $1$.

In the remaining case, $\tau_+$ has a hat at position $i > 0$, and we
must move the hat left one step. If the label to the left is smaller
than the hatted label, then we move the hat left with probability $1$
and are done. But if that label is larger, then with probability $1/q$, 
we move the label bearing the hat, not just the hat.
From the state $3\ - 2\ \hat 1$:
$$ 3\ - \hat 2\ 1   \qquad \stackrel{1-q^{-1}}\longrightarrow \qquad
3\ - 2\ \hat 1 \qquad \stackrel{q^{-1}}\longleftarrow \qquad
3\ -\hat 1\ 2 
$$
(Don't forget that the Markov chain runs backwards along the arrows.)

Again, we claim that if $\tau, \tau'$ are unhatted states, 
and we start this Markov chain at $\tau$ and stop when we next
meet an unhatted state, the probability that it is $\tau'$
is the same as that given by the Markov chain from \S \ref{sec:richer}.

\junk{

\subsection{The hatted linear algebra}

We need to hat the elements of $U_b$. Given $M \in U_b$, 
let $N_M \in \naturals$ denote the position of the last pivot column in $M$. 
The elements of $U_b^+$ will be equivalence classes 
of pairs $(M \in U_b, i\in [0, 1 + N_M])$ where $M$'s $i$th column is 
only recorded modulo the span of the columns to its left. 
Now we need a Markov chain on the states $U_b \coprod U_b^+$ that projects
to the one in \S \ref{ssec:fiddlychain}. 

If $M \in U_b$, its only transition is to $(M, N_M+1) \in U_b^+$ where $N_M$ 
is as above. (Note that column $N_M+1$ is completely indeterminate.)
Whereas $(M,0) \in U_b^+$ transitions to $M \in U_b$
(conversely, column $0$ is well-defined).
In the remaining case $(M,i>0,\vec c)$, we pick a lift 
$\vec c' \in \FF^b / span[0,i-1)$ with $\vec c' \mapsto \vec c$, 
and decrement $i$.  The choice $c'$ is unique except when $i-1$ is a
pivot column of $M$, in which case there are $q$ choices.

Informally, instead of having $\vec c$ appear at the front of $M$ where we
can see it clearly, it wanders in from the back, gradually coming into view.
In this way it obviously factors the Markov chain from \S \ref{sec:motivation}.

Finally, we need a map $\sigma_+$ from $U_b^+$ to the set of 
hatted juggling states. It takes $(M,i,\vec v)$ to $\sigma(M)$
hatted at $i$ with ... the number of pivot columns 

}

\section{Questions}

What is a linear algebra interpretation of the model in \S \ref{sec:hats}?

Is there an analogue of \cite{AMM} for the flag juggling digraph?

The stationary distributions computed here live on 
$S_{b+\naturals} / (S_b \times S_\naturals)$ and 
$S_{b+\naturals} / ((S_1)^b \times S_\naturals)$,
and easily generalize to other $W/W_P$ coset spaces
($\tau \mapsto q^{-\ell(\tau)}$ times a prefactor).
What is a Markov chain on $W/W_P$
for which they are the stationary distribution?

What is an analogue of the ball density function
calculated in \S \ref{sec:stats}, when the balls are colored
as in \S \ref{sec:richer}?

Is there a version of \S \ref{sec:intro} in which $b$ varies, whose
stationary distribution is still $q^{-\text{area}}$, up to overall
scale? Does that have a linear algebra interpretation?

Is the $q\to 1$ limit related in any substantive way to the mythical field 
of $1$ element?  Can that theoretical theory include the $b\to\infty$ limit?

It is easy to define a version of the digraph from \S \ref{sec:hats}
with multiple hatted labels percolating through independently. 
Is this of any use?

\bibliographystyle{alpha}    

\end{document}